\documentclass{amsart}
\usepackage{chngcntr}

\newif\ifdraft\draftfalse
\newif\ifcite%\citetrue

\ifdraft\ifcite\usepackage{showkeys}\else\usepackage[notcite,notref]{showkeys}\fi\fi
\usepackage{amsmath, amsthm, amssymb, amsfonts, eucal}

\makeatletter

\theoremstyle{plain}

\newtheorem{theorem}[equation]{Theorem}

\newtheorem{lemma}[equation]{Lemma}

\newtheorem{corollary}[equation]{Corollary}

\theoremstyle{remark}

\theoremstyle{definition}
\newtheorem{definition}[equation]{Definition}

\theoremstyle{remark}

\DeclareMathOperator{\Sym}{Sym}
\DeclareMathOperator\Reg{\mathrm{Reg}}
\DeclareMathOperator\Gr{\mathrm{Gr}}

\DeclareMathOperator{\Orth}{\mathbf{O}}

\DeclareMathOperator\vol{\mathrm{vol}}
\DeclareMathOperator\diam{\mathrm{diam}}

\makeatletter
\newcommand{\leqnomode}{\tagsleft@true}
\newcommand{\reqnomode}{\tagsleft@false}
\makeatother

\newcommand\CC{\mathbb C}

\begin{document}
\title[Volumes and affine GAGA]{Volumes of definable sets in o-minimal expansions
and affine GAGA theorems}
\author{Patrick Brosnan}
\address{Department of Mathematics\\
  University of Maryland\\
  College Park, MD USA}
\email{pbrosnan@umd.edu}

\begin{abstract}
  In this mostly expository note, I give a very quick proof of the 
  definable Chow theorem of Peterzil and Starchenko using the Bishop-Stoll
  theorem and a volume estimate for definable sets due to Nguyen and 
  Valette. 
  The volume estimate says that any $d$-dimensional definable subset of $S\subseteq\mathbb{R}^n$ 
  in an o-minimal
  expansion of the ordered field of real numbers  satisfies the inequality
  $\mathcal{H}^d(\{x\in S:\lVert x\rVert<r\})\leq Cr^d$, where $\mathcal{H}^d$
  denotes the $d$-dimensional Hausdorff measure on $\mathbb{R}^n$ and $C$ is a
  constant depending on $S$.  
  A closely related volume estimate for subanalytic sets goes back to Kurdyka and Raby.
  Since this note is intended to be helpful to algebraic geometers not versed in o-minimal structures
  and definable sets, I review these notions and also prove the main volume estimate
  from scratch.
\end{abstract}
\maketitle

\section{Introduction} The GAGA theorem of Peterzil--Starchenko~\cite{ps-ns}
says that a closed analytic subset of $\CC^n$, which is definable in an
o-minimal expansion of the ordered field $\mathbb{R}$, is algebraic.  It is a
crucial ingredient in (at least) two, closely related, recent advances in Hodge theory:
the paper by Bakker, Klingler and Tsimerman~\cite{bkt}, 
which gives a new proof of the
theorem of Cattani, Deligne and Kaplan~\cite{cdk} 
on the 
algebraicity
of the Hodge locus,
and the paper by Bakker, Brunebarbe and Tsimerman proving a conjecture of
Griffiths on the algebraicity of the image of the period map~\cite{bbt}.  
Given these important results, it seems desirable to have an understanding of
the Peterzil--Starchenko theorem from several points of view.  

The point of this (mostly expository) note is to show that the
Peterzil--Starchenko theorem follows directly from a GAGA theorem originally
due to Stoll~\cite{stoll} and a volume estimate for definable sets.
This volume estimate was known to experts in o-minimal structures for 
several years now: it is a special case of Proposition 3.1 of a 2018 paper by 
Nguyen and Valette~\cite{nv}.
Moreover, in the context of subanalytic sets, it follows from a
paper of Kurdyka and Raby~\cite{KurdykaRaby}.
I will state it precisely in Theorem~\ref{volest} below 
and prove it from scratch in \S\ref{s.proof},
but essentially it says the following:  Suppose $S$ is a $d$-dimensional subset
of $\mathbb{R}^n$, which is definable with respect to an o-minimal expansion of
$\mathbb{R}_{\mathrm{alg}}$ (for example,
$\mathbb{R}_{\mathrm{an},\mathrm{exp}}$). Then the Hausdorff measure of the set
$S(r):=\{x\in S:\lVert x\rVert< r\}$ 
viewed as a function of $r$ is in $O(r^d)$.
In other words, the volume of the intersection of $S$ with a ball of
radius $r$ grows at most as fast as a constant multiple of $r^d$.
In~\S\ref{sec-ps}, I use it to give a very quick proof of the Peterzil-Starchenko theorem.

Peterzil and Starchenko published two proofs of their theorem.  
The first, in~\cite{ps-ns}, works for o-minimal expansions of arbitrary real
closed fields and is based on results from model theory. The second proof,
in~\cite{ps-crelle}, like the proof presented in this note, relies on results
from complex analysis.  
More precisely, it relies on a paper of Shiffman that is directly related to
Stoll's theorem~\cite{shiffman} in that Shiffman's results are ultimately about
bounds on volumes of complex analytic sets.  
Aside from brevity, the main advantage of my approach is that it makes it clear
that the proof uses volume estimates that hold for all definable sets, 
not just complex analytic sets.
Still, while I think the viewpoint and the brevity of this paper are
worthwhile, the techniques are similar to the techniques of~\cite{ps-crelle}.

The proof of Theorem~\ref{volest} given in this paper is also very 
similar to the proof given by Nguyen and Valette in~\cite{nv}, which 
I learned about after the first version of this paper appeared on the ArXiv.
I decided to keep my original proof of Theorem~\ref{volest} in this note because it is self-contained
and the exposition is aimed at readers who are not experts in o-minimal
structures.  
However, the reader should be aware that neither Theorem~\ref{volest} nor its 
proof are new.
The main claim to novelty in this paper is that it points out that the version of Peterzil-Starchenko's 
definable Chow theorem proved in~\cite{ps-crelle} has an easy proof 
using Stoll's theorem.
This is interesting because it indicates that the two theorems are directly related.

To help make this paper approachable for algebraic geometers, I review
the theory of o-minimal structures in section \S\ref{sec-omin}.
In~\S\ref{sec-bs}, I review the notions of Hausdorff dimension and
Stoll's theorem, which is also sometimes called the Bishop--Stoll
theorem.  
(Bishop~\cite{bishop} generalized and extended the result of Stoll 
used in this paper.)  
In \S\ref{sec-voldef}, I state the main volume estimate, Theorem~\ref{volest}.

\subsection{Acknowledgments} I thank my colleague Chris Laskowski for several
enlightening conversations about model theory in general and o-minimal
structures in particular.  
I am also happy to thank Najmuddin Fakhruddin, Priyankur Chaudhuri, Eoin
Mackall and Swarnava Mukhopadhyay for corrections they pointed out after
reading  previous versions of this paper.
Finally, I am extremely grateful to 
  Xuan Viet Nhan Nguyen,
  who emailed me to point out that my Theorem~\ref{volest}
  is a special case of~\cite[Proposition 3.1]{nv}, and to
  Tobias Kaiser, who emailed to point out the relationship
  with the (considerably older) work of Kurdyka and Raby~\cite{KurdykaRaby}.

  \section{o-minimal structures}\label{sec-omin}
  In defining o-minimal structures, I follow the book by van den Dries~\cite{vdd}.

  \begin{definition}\label{d-omin}
    An o-minimal structure on $\mathbb{R}$ is a sequence
    $\mathcal{S}=(\mathcal{S}_n)_{n\in\mathbb{N}}$ of sets such that, for each $n$:
    \begin{enumerate}
    \item $\mathcal{S}_n$ is a boolean algebra of subsets of $\mathbb{R}^n$;
    \item $A\in\mathcal{S}_n$ implies that  $A\times\mathbb{R}$ and 
      $\mathbb{R}\times A$ are in $\mathcal{S}_{n+1}$;
    \item If $1\leq i<j\leq n$, then
      $\{(x_1,\ldots, x_n)\in\mathbb{R}^n:x_i=x_j\}\in\mathcal{S}_n$.
    \item If $\pi:\mathbb{R}^{n+1}\to\mathbb{R}^n$ denotes the projection
      onto a factor, then $A\in\mathcal{S}_{n+1}\implies \pi(A)\in\mathcal{S}_n$;
    \item For each $r\in\mathbb{R}$, $\{r\}\in\mathcal{S}_1$.  Moreover, 
      $\{(x,y)\in\mathbb{R}^2:x<y\}\in\mathcal{S}_2$;
    \item The only subsets in $\mathcal{S}_1$ are the finite unions
      of intervals and points.  
    \end{enumerate}
  \end{definition}

  Call a sequence $\mathcal{S}$ a
  \emph{structure} if it satisfies all of the hypotheses of
  Definition~\ref{d-omin} except possibly the last two~\cite[p.13]{vdd}.
  If $X\in \mathcal{S}_n$ for some $n$, then we say that $X$ is a
  \emph{definable} subset of $\mathbb{R}^n$ with respect to the
  structure $\mathcal{S}$.  Similarly, if $f:X\to Y$ is a function
  with $X\subset\mathbb{R}^n$ and $Y\subset \mathbb{R}^m$ for $n,m\in\mathbb{N}$,
  then we say $f$ is \emph{definable} if its graph (viewed as a subset of 
  $\mathbb{R}^n\times\mathbb{R}^m=\mathbb{R}^{n+m}$) is definable.  

  It is clear that, if we let $\mathcal{S}_n=\mathcal{P}(\mathbb{R}^n)$, i.e.,
  the power set of $\mathbb{R}^n$, then we get a structure.  
  (But obviously not an o-minimal one.)  
  It is also relatively easy to see that the intersection of structures is a
  structure.  
  So, given an arbitrary collection
  $\mathcal{T}_n\subset\mathcal{P}(\mathbb{R}^n)$ (for $n\in\mathbb{N}$), there
  is a smallest structure $(\mathcal{S}_n)_{n\in\mathbb{N}}$ containing
  $\mathcal{T}_n$.   
  This is the structure \emph{generated} by the $\mathcal{T}_n$.

  If $\{\mathcal{S}_n\}$ and $\{\mathcal{S}'_n\}$ are both structures with
  $\mathcal{S}'_n\subset \mathcal{S}_n$ for all $n$, then $\{\mathcal{S}_n\}$ is
  called an \emph{expansion} of $\{\mathcal{S}'_n\}$.  If
  $\{\mathcal{T}_n\}$ is any collection with
  $\mathcal{T}_n\subset\mathcal{P}(\mathbb{R}^n)$, the \emph{structure generated
  by $\{\mathcal{S}_n\cup\mathcal{T}_n\}_{n\in\mathbb{N}}$} is called the
  \emph{expansion of $\{\mathcal{S}_n\}$ generated by $\{\mathcal{T}_n\}$}. 

  One classical example of an o-minimal structure on $\mathbb{R}$ is the structure
  $\mathbb{R}_{\mathrm{alg}}$ consisting of all semi-algebraic sets.  (The 
  fact that $\mathbb{R}_{\mathrm{alg}}$ satisfies (iv) is the content of the
  Tarski-Seidenberg theorem.)  The example
  that is most important for the recent applications to Hodge theory mentioned in
  the introduction is the one called $\mathbb{R}_{\mathrm{an},\mathrm{exp}}$.
  This is the expansion of $\mathbb{R}_{\mathrm{alg}}$ generated by the graph of
  the real exponential function $x\mapsto e^x$ and the collection of all graphs
  of analytic functions on $[0,1]$.  (See
  ~\cite{vdd-m} for references. The o-minimality 
  of the expansion $\mathbb{R}_{\mathrm{exp}}$ of $\mathbb{R}_{\mathrm{alg}}$
  generated by the graph of the real exponential function is 
  a celebrated theorem of Wilkie~\cite{wilkie}.)

  It is convenient to think about definable sets in a structure in terms
  of logic as subsets of $\mathbb{R}^n$ defined by the formulas in a
  language $\mathcal{L}$ interpreted in the field of real numbers.  This
  point of view is explained (a little informally) in~\cite[Chapter
  1]{vdd}.  (For a more precise explanation of the model theory point of
  view, see, for example, Marker's book~\cite{marker}).  
  Here subsets
  $\psi$ of $\mathbb{R}^n$ are thought of as properties
  $\psi(x_1,\ldots, x_n)$ of $n$-tuples $(x_1,\ldots, x_n)$ of real
  numbers with $\psi(x_1,\ldots, x_n)$ being the property that $(x_1,\ldots, x_n)\in \psi$.  
  Suppose $S=\{\psi_i\}$ is a collection of such subsets, with
  $\psi_i\subset\mathbb{R}^{n_i}$.  
  Then the expansion of $\mathbb{R}_{\mathrm{alg}}$ generated by $\psi$
  consists of the subsets of $\mathbb{R}^n$ definable by formulas involving
  the field operations on $\mathbb{R}$, the real numbers (viewed as constants),
  the symbols $<$ and $=$, variables $(x_i)_{i=1}^{\infty}$, and the  $\psi_i$
  along with the $\forall, \exists$ and the usual logical connectives.

  \section{Volumes and the Bishop--Stoll Theorem}\label{sec-bs}

  My main reference for this section is G.~Stolzenberg's
  book~\cite{stvol}.

  Let $X=(X,d_X)$ be a metric space.
  For $\emptyset\neq S\subset X$, the diameter of $S$ is
  $\diam S:=\sup\{d_X(x,y):x,y\in S\}$.  By convention, write
  $\diam\emptyset =-\infty$.

  Suppose $S\subset X$, and $\epsilon$ is a positive real number.
  An $\epsilon$-covering of $S$ is a countable collection $\{S_i\}_{i=1}^{\infty}$
  of subsets of $S$ of diameter less than $\epsilon$ such that
  $S\subset \cup_{i=1}^{\infty} S_i$.
  Fix a non-negative real number $d$ and set 
  $$
  I(d,\epsilon, S):=\inf\left\{\sum_{i=1}^{\infty} (\diam S_i)^d:
  \{S_i\}_{i=1}^{\infty}\text{ is an $\epsilon$-covering of }S\right\}.$$
  The \emph{$d$-Hausdorff measure of $S$} is
  $$
  \mathcal{H}_d(S):=\frac{1}{2^d}\lim_{\epsilon\to 0^+} I(d,\epsilon, S).
  $$

  If $d$ is a non-negative integer and $S$ is a closed $d$-dimensional
  sub-manifold of $\mathbb{R}^n$, then the volume $\vol_d S$ (defined in the
  usual way with respect to the standard metric on $\mathbb{R}^n$) is given by
  \begin{equation} \label{eq:vold} \vol_d(S)= \frac{\pi^{d/2}}{\Gamma(\frac{d}{2}
  +1)}\mathcal{H}_d(S).  \end{equation} As it turns out, we want to refer to
  volume instead of $\mathcal{H}_d(S)$ in general.  So we use 
  the equation~\eqref{eq:vold} as a definition to define 
  $\vol_d(S)$ for an arbitrary non-negative real
  number $d$.   
  (Note that Federer's normalization for Hausdorff measure   
  in his book~\cite[\S2.10.2]{federer} differs from that of Stolzenberg.  
  For Federer, the $d$-dimensional Hausdorff measure $\mathcal{H}_d(S)$ is just what 
  we call $\vol_d(S)$.)

  Let's also make the convention that we always regard a subset $S\subset\mathbb{R}^n$
  as a metric subspace of $\mathbb{R}^n$ with its standard metric.
  For a positive real number $r$, set
  $B(r)=B_{n}(r):=\{x\in\mathbb{R}^n:|x|<r\}$.  Then, if $S\subset\mathbb{R}^n$,
  set 
  \[
    S(r):=S\cap B(r).
  \]
  We use the Big-O notation: if $f,g$ are two real valued functions defined on an
  interval of the form $(a,\infty)$, then we write $f=O(g)$ if there exists a
  constant $C$ and a real number $b>a$ such that 
  $$
  x>b\Rightarrow |f(x)|\leq Cg(x).
  $$

  \begin{theorem}[Stoll]\label{thm-bs}  Suppose that $Z$ is a closed analytic subset of
    $\CC^n$ of pure dimension $d$.  If $\vol_{2d} Z(r)=O(r^{2d})$,
    then $Z$ is algebraic.
  \end{theorem}

  See~\cite[p.~2, Theorem D]{stvol} for the statement.  A proof is given
  in~\cite[Chapter IV]{stvol}.  
  According to Cornalba and Griffiths~\cite[E.
  4.2]{crgr},  the converse also holds.  
  (This also follows directly from the main result of this note, Theorem~\ref{volest} below.)

  \section{Volumes of Definable Sets}\label{sec-voldef}

  My main goal in this letter is to prove the Peterzil--Starchenko GAGA theorem
  (Theorem~\ref{thm-ps} below) using Theorem~\ref{thm-bs} and a general fact
  about definable sets and Hausdorff measures.  
  To explain this general fact, let me first explain cells.  To do
  this, fix an o-minimal expansion $\mathbb{R}_{\mathrm{alg},*}$  of
  $\mathbb{R}_{\mathrm{alg}}$. 

  \subsection{Cells}  These are certain special
  definable subsets (with respect to $\mathbb{R}_{\mathrm{alg},*}$) of
  $\mathbb{R}^n$ defined inductively.  See page 50 of~\cite{vdd} for a
  complete definition, but, roughly speaking, cells in $\mathbb{R}^n$
  are defined inductively with respect to $n$ as either
  \begin{enumerate}
  \item[(a)] graphs of continuous definable functions $f$ on
    cells in $\mathbb{R}^{n-1}$ or,
  \item[(b)] nonempty open regions in between graphs of continuous definable
  functions in $\mathbb{R}^{n-1}$.  
  \end{enumerate} 
  There is a dimension function $d$ defined inductively on the set of all cells
  by setting $d(S)=d(T)$ if $S\subset\mathbb{R}^n$ is constructed inductively
  from $T\subset\mathbb{R}^{n-1}$ via procedure (a) and setting $d(S)=d(T)+1$ if
  it is constructed via procedure (b).  Moreover, as a consequence of the cell
  decomposition theorem~\cite[2.11 on p.~52]{vdd}, every definable subset of
  $\mathbb{R}^n$ can be written as a finite disjoint union of cells.  (This is
  one of the most crucial properties of definable sets in o-minimal structures.)
  If $X$ is then any definable subset of $\mathbb{R}^n$, van den Dries defines
  $\dim X$ to be the maximum of the dimensions $d(S)$ as $S$ ranges over all
  cells contained in $X$~\cite[p.~63]{vdd}.  Since, by~\cite[p.~64]{vdd}, $\dim
  (X\cup Y)=\max(\dim X,\dim Y)$, $\dim X$ is also the maximum of the dimensions
  $d(S)$ of the cells $S$ appearing in a decomposition of $X$ into disjoint
  cells.

  Now, I am ready to state the main volume estimate of this paper. 
  As mentioned in the introduction, this theorem is a special
  case of Proposition 3.1 of~\cite{nv}.
  Moreover, the paper~\cite{KurdykaRaby} by Kurdyka and Raby proves
  an equivalent result in the language of subanalytic subsets.

\begin{theorem}\label{volest}
  Suppose $S\subset\mathbb{R}^n$ is a set  which is definable in an 
  o-minimal expansion of $\mathbb{R}_{\mathrm{alg}}$.  Set $d=\dim S$.   Then
  $$
  \vol_d S(r)=O(r^d).
  $$
\end{theorem}

\subsection{Peterzil--Starchenko}\label{sec-ps}  Before proving
Theorem~\ref{volest}, I want to use it to prove the Peterzil--Starchenko
GAGA theorem.

\begin{theorem}[Peterzil--Starchenko]
\label{thm-ps}Let $A$ be a closed, complex analytic subset of $\CC^n$, which is
definable with respect to an o-minimal expansion of
$\mathbb{R}_{\mathrm{alg}}$.  
Then $A$ is an algebraic subset of $\CC^n$. 
\end{theorem}

\begin{proof}[Proof of Theorem~\ref{thm-ps} using Theorem~\ref{volest}]
  Take a closed complex analytic subset $Z\subset\CC^n$ of complex dimension
  $d$, and assume $Z$ is definable. Then $Z$ has a definable dense open subset 
  $U$ which is submanifold of $\CC^n$.  It follows that the dimension
  of $Z$ as a definable subset of $\CC^n$ (in the 
  sense of~\cite[Definition 4.1.1]{vdd}) is $2d$. Then, by Theorem~\ref{volest}, we
  have $\vol_{2d} Z(r)=O(r^{2d})$. So, by Theorem~\ref{thm-bs}, $Z$ is algebraic.
\end{proof}

In the next section, I prove Theorem~\ref{volest}. 
The elementary proof mainly relies on the Gauss map and change of variables.

\section{Proof of the Volume Estimate}\label{s.proof}
In what follows it will be convenient to note that, since linear transformations between
finite dimensional real vector spaces are definable in any expansion of 
$\mathbb{R}_{\mathrm{alg}}$, any finite dimensional real vector space $V$ 
comes equipped with a canonical definable structure (via any linear isomorphism 
to $\mathbb{R}^{\dim V}$).     

If $X$ is a definable subset of $\mathbb{R}^n$ of dimension $d$, we say that
\emph{Theorem~\ref{volest} holds for $X$} if $\vol_d X(r)=O(r^d)$.  
Note that, if Theorem~\ref{volest} holds for $X$, then, for $d'>d$, we have
$\vol_{d'} X(r)=0$ for all $r$. 
(See~\cite[\S2.10.2]{federer}.)
For each nonnegative integer $d$, write $P(d)$ for the assertion that
Theorem~\ref{volest} holds for all definable sets $X$ of dimension $\leq d$.
The goal of the section is to prove $P(d)$ for each nonnegative integer $d$
by induction on $d$.  
Since zero dimensional definable sets are finite, $P(0)$
obviously holds.

Write $\Gr(d,n)$ for the Grassmannian of real $d$-dimensional planes
through the origin in $\mathbb{R}^n$.  
The set $\Gr(d,n)$ has a natural structure of a definable $C^{\infty}$-manifold.
In the language of~\cite[Chapter 10]{vdd}, $\Gr(d,n)$ is a definable space, 
which is also (compatibly) a compact $C^{\infty}$-manifold.  
(See also ~\cite{FischerAdvMath} for a precise definition.)
As it is also a regular space,
\cite[Theorem 10.1.8]{vdd} implies that it is isomorphic (as a definable space)
to an affine definable space, i.e., a definable subset of $\mathbb{R}^n$.
However, the method used in~\cite[Example 10.1.4]{vdd} to show that 
$\mathbb{P}^n(\mathbb{R})$ is affine, can be imitated to realize $\Gr(d,n)$
as a closed definable $C^{\infty}$-submanifold of $\mathbb{R}^N$ for some 
suitable $N$.

To be explicit about this last point, 
endow $\mathbb{R}^n$ with the usual dot product.  
For each $L\in\Gr(d,n)$, choose an ordered basis $\tilde L=(\ell_1,\ldots, \ell_d)$, and
set $\wedge^d \tilde L:=\ell_1\wedge\cdots\wedge\ell_d$.  
Then let $\phi(L)$ denote the point
\[
  \phi(L):=\frac{1}{\lVert \wedge^d \tilde L\lVert^2}(\wedge^d \tilde L)\otimes (\wedge^d \tilde L)
           \in \Sym^2(\wedge^d \mathbb{R}^n).
\]
The resulting map $\phi:\Gr(d,n)\to \Sym^2(\wedge^d \mathbb{R}^n)$ is a well-defined
and definable smooth morphism, embedding $\Gr(d,n)$ as a closed, definable, smooth
submanifold of $\Sym^2(\wedge^d\mathbb{R}^n)$.  
So we can identify $\Gr(d,n)$ with the image of $\phi$. 

We can also view $\Gr(d,n)$ as a quotient of the orthogonal group
$\Orth(n)$ in the usual way.  
Note that the 
orthogonal group $\Orth(n)$ (of real $n\times n$-matrices 
which are orthogonal with respect to the standard inner product) is a closed
definable and smooth submanifold of $\mathbb{R}^{n^2}$.
Moreover, it is a group in the category of definable spaces. 
It acts definably, properly and transitively on the space $\Gr(d,n)$.   
The stabilizer of $L\in\Gr(d,n)$ is the definable, closed subgroup 
$\Orth(L)\times\Orth(L^{\perp})$, 
which is definably isomorphic to $\Orth(d)\times\Orth(n-d)$.   
From this, it is not hard to see that, 
in the language of~\cite[p.~162]{vdd}, 
$\Gr(d,n)$ is a definably proper quotient of $\Orth(n)$, 
and, in fact,
$\Gr(d,n)$ 
is definably isomorphic to the quotient $\Orth(n)/(\Orth(d)\times\Orth(n-d))$. 

If $L\in\Gr(d,n)$ is a $d$-dimensional linear subspace,
we write $\pi_L:\mathbb{R}^n\to L$ for the orthogonal projection onto
the subspace $L$.  
The map
$z\mapsto (\pi_L(z),\pi_{L^{\perp}}(z))$ 
is then an isometric (and, thus, volume-preserving) definable isomorphism from 
$\mathbb{R}^n$ to $L\times L^{\perp}$. 

\begin{lemma}\label{l.volcell}
  Suppose $L\in\Gr(d,n)$. There exists a definable open neighborhood $U_L$ of $L$
  in $\Gr(d,n)$ such that the following two statements hold:
  \begin{enumerate}
  \item $\pi_L(L')=L$ for all $L'\in U_L$.
  \item Suppose $C=\{(x,y)\in D\times L^{\perp}: y=f(x)\}$ where
    $D\subset L$ is a definable $d$-dimensional cell and $f:D\to L^{\perp}$
    is a $C^2$ definable function.  Assume that the tangent
    space $T_zC$ is in $U_L$
    for all $z=(x,y)\in C$.  Then, for all $r\in\mathbb{R}$,  
    $\vol_d C(r)\leq 2\vol_d D(r)$.  
  \end{enumerate}
\end{lemma}

\begin{proof}
  Since the group $\Orth(n)$ acts transitively on $\Gr(d,n)$ and 
  preserves the metric (and hence the volume form) on $\mathbb{R}^n$, 
  we can assume $L=\mathbb{R}^d=\{(x_1,\ldots, x_n)\in\mathbb{R}^n:
  x_i=0$ for $i=d+1,\ldots, n\}$.  
  Write $e_i$ for the tangent vector
  $\partial/\partial x_i$, and write $\Phi:D\to\mathbb{R}^n$ for the
  map $x\mapsto (x,f(x))$.  
  Let  $A(x)=Df(x)$ denote the derivative of $f$ at $x$.  
  For $z=(x,f(x))\in C$, the tangent space $T_z C$ is the $d$-dimensional space
  generated by the vectors $v_i=e_i+A(x)e_i\in\mathbb{R}^n$.  
  For $r\in\mathbb{R}$, $\pi_L(C(r))\subset D(r)$.  
  Therefore, 
  $\vol_d C(r)\leq\int_{D(r)} \sqrt{\det g_{ij}(x)}\, dx$ 
  where $g_{ij}(x)$ is the matrix 
  $v_i\cdot v_j=\delta_{ij}+e_i\cdot A(x)^*A(x)e_j$.  
  (See~\cite[\S3.2.46]{federer} for the relevant formula computing the  volume 
  in terms of $g_{ij}$.) 
  For $T_z C$ sufficiently close to $L$, the matrix $A(x)$ will be close to $0$.  
  Therefore, $g_{ij}(x)$ will be close to the identity matrix.  
  From these considerations, the lemma follows easily.  
\end{proof}

  For the rest of the section, pick definable neighborhoods $U_L$ for each
  $L\in\Gr(d,n)$ once and for all.  
  For each definable set $X$, we let $\Reg^2
  X$ denote the locus in $X$ consisting of all points $x\in X$ such that
  there is an open subset $U$ of $\mathbb{R}^n$ such that $U\cap X$ is a
  $C^2$-manifold.  
  This is a definable subset of $X$, and the complement
  $X\setminus\Reg^2 X$ has dimension strictly less than the dimension of
  $X$.
  (This follows from the Cell Decomposition Theorem of \cite[\S 4.2]{vdd-m}.)

\begin{corollary}\label{c.volflat}
  Suppose $L\in\Gr(d,n)$ and $M$ is a $d$-dimensional $C^2$ definable
  submanifold of $\mathbb{R}^n$ such that $T_xM\in U_L$ for each $x\in M$.
  Then, assuming that $P(k)$ holds for $k<d$, we have
  $\vol_d M(r)=O(r^d)$.
\end{corollary}
\begin{proof}
  We can write $M$ as a finite union of cells of dimension $\leq d$.
  Moreover, we can assume that $L=\{x\in\mathbb{R}^n:x_i=0$ for $i>d\}$. 
  Then, using the assumption that $P(k)$ holds for $k<d$, we can
  assume that $M$ is, in fact, a single cell.  
  Since $D\pi_L(x):T_xM\to L$ is onto for all $x\in M$, 
  we see easily that $M$ has the form of the subset 
  $C$ in Lemma~\ref{l.volcell}.  
  So $M=\{(x,y)\in D\times L^{\perp}:
  y=f(x)\}$ with $f$ and $D$ as in Lemma~\ref{l.volcell}. 
  As $D$ is a $d$-dimensional cell in $\mathbb{R}^d$, it is open.  
  So it is obvious that $\vol_d D(r)=O(r^d)$.   
  Then the volume estimate for $M$ follows from Lemma~\ref{l.volcell}.
\end{proof}

\begin{proof}[Proof of Theorem~\ref{volest}]
  Suppose $P(k)$ holds for $k<d$ and that
  $S$ is a $d$-dimensional definable set.  We can write $S$ as a finite
  union of $C^2$-cells, so, since we are assuming $P(k)$ holds for $k<d$,
  we can assume that $S$ is itself a $C^2$ cell. In particular $S$ is
  a $C^2$ submanifold of $\mathbb{R}^n$.

  Write $\Gamma:S\to\Gr(d,n)$
  for the Gauss map sending $x\in S$ to $T_xS$.  Then $\Gamma$ is definable
  and continuous.  So cover $\Gr(d,n)$ with finitely many opens of the
  form $U_{L_i}$ for $i=1,\ldots, m$.  Then each set $S_i:=\Gamma^{-1}(U_{L_i})$
  is definable open in $S$.   
  In particular, $S_i$ is a $C^2$ submanifold
  of $\mathbb{R}^n$ and $T_xS_i\in U_{L_i}$ for all $i$ and for each $x\in S_i$.   
  The fact that $
  \vol_d S_i(r)=O(r^d)$ then follows form Corollary~\ref{c.volflat}.
  Since $m<\infty$, this proves the theorem. 
\end{proof}

% \printbibliography
\bibliographystyle{plain}
%% input <logic.bbl>
%% \bibliography{logic}

\begin{thebibliography}{10}

\bibitem{bkt}
B.~Bakker, B.~Klingler, and J.~Tsimerman.
\newblock Tame topology of arithmetic quotients and algebraicity of {H}odge
  loci.
\newblock {\em J. Amer. Math. Soc.}, 33(4):917--939, 2020.

\bibitem{bbt}
Benjamin {Bakker}, Yohan {Brunebarbe}, and Jacob {Tsimerman}.
\newblock {o-minimal GAGA and a conjecture of Griffiths}.
\newblock {\em arXiv e-prints}, November 2018.

\bibitem{bishop}
Errett Bishop.
\newblock Conditions for the analyticity of certain sets.
\newblock {\em Michigan Math. J.}, 11:289--304, 1964.

\bibitem{cdk}
Eduardo Cattani, Pierre Deligne, and Aroldo Kaplan.
\newblock On the locus of {H}odge classes.
\newblock {\em J. Amer. Math. Soc.}, 8(2):483--506, 1995.

\bibitem{crgr}
Maurizio Cornalba and Phillip Griffiths.
\newblock Analytic cycles and vector bundles on non-compact algebraic
  varieties.
\newblock {\em Invent. Math.}, 28:1--106, 1975.

\bibitem{federer}
Herbert Federer.
\newblock {\em Geometric measure theory}.
\newblock Die Grundlehren der mathematischen Wissenschaften, Band 153.
  Springer-Verlag New York Inc., New York, 1969.

\bibitem{FischerAdvMath}
Andreas Fischer.
\newblock Smooth functions in o-minimal structures.
\newblock {\em Adv. Math.}, 218(2):496--514, 2008.

\bibitem{KurdykaRaby}
K.~Kurdyka and G.~Raby.
\newblock Densit\'{e} des ensembles sous-analytiques.
\newblock {\em Ann. Inst. Fourier (Grenoble)}, 39(3):753--771, 1989.

\bibitem{marker}
David Marker.
\newblock {\em Model theory}, volume 217 of {\em Graduate Texts in
  Mathematics}.
\newblock Springer-Verlag, New York, 2002.

\bibitem{nv}
Nhan Nguyen and Guillaume Valette.
\newblock Whitney stratifications and the continuity of local
  {L}ipschitz-{K}illing curvatures.
\newblock {\em Ann. Inst. Fourier (Grenoble)}, 68(5):2253--2276, 2018.

\bibitem{ps-ns}
Ya'acov Peterzil and Sergei Starchenko.
\newblock Complex analytic geometry in a nonstandard setting.
\newblock In {\em Model theory with applications to algebra and analysis.
  {V}ol. 1}, volume 349 of {\em London Math. Soc. Lecture Note Ser.}, pages
  117--165. Cambridge Univ. Press, Cambridge, 2008.

\bibitem{ps-crelle}
Ya'acov Peterzil and Sergei Starchenko.
\newblock Complex analytic geometry and analytic-geometric categories.
\newblock {\em J. Reine Angew. Math.}, 626:39--74, 2009.

\bibitem{shiffman}
Bernard Shiffman.
\newblock On the removal of singularities of analytic sets.
\newblock {\em Michigan Math. J.}, 15:111--120, 1968.

\bibitem{stoll}
Wilhelm Stoll.
\newblock The growth of the area of a transcendental analytic set. {I}, {II}.
\newblock {\em Math. Ann.}, 156:144--170, 1964.

\bibitem{stvol}
Gabriel Stolzenberg.
\newblock {\em Volumes, limits, and extensions of analytic varieties}.
\newblock Lecture Notes in Mathematics, No. 19. Springer-Verlag, Berlin-New
  York, 1966.

\bibitem{vdd}
Lou van~den Dries.
\newblock {\em Tame topology and o-minimal structures}, volume 248 of {\em
  London Mathematical Society Lecture Note Series}.
\newblock Cambridge University Press, Cambridge, 1998.

\bibitem{vdd-m}
Lou van~den Dries and Chris Miller.
\newblock Geometric categories and o-minimal structures.
\newblock {\em Duke Math. J.}, 84(2):497--540, 1996.

\bibitem{wilkie}
A.~J. Wilkie.
\newblock Model completeness results for expansions of the ordered field of
  real numbers by restricted {P}faffian functions and the exponential function.
\newblock {\em J. Amer. Math. Soc.}, 9(4):1051--1094, 1996.

\end{thebibliography}
%%%%% BEGIN logic.bbl%%%%%%%%%%%%%%%

%%%%% END  logic.bbl%%%%%%%%%%%%%%%
\end{document}